\documentclass[12pt, reqno]{amsart}
\usepackage{amsmath, amsthm, amscd, amsfonts, amssymb, graphicx, color}
\usepackage[bookmarksnumbered, colorlinks, plainpages]{hyperref}
\hypersetup{colorlinks=true,linkcolor=red, anchorcolor=green, citecolor=cyan, urlcolor=red,
filecolor=magenta, pdftoolbar=true}

\newtheorem{theorem}{Theorem}[section]
\newtheorem{lem}[theorem]{Lemma}
\newtheorem{thm}[theorem]{Theorem}

\theoremstyle{definition}
\newtheorem{defn}[theorem]{Definition}
\newtheorem{ex}[theorem]{Example}

\theoremstyle{remark}

\numberwithin{equation}{section}
\newcommand{\R}{{\mathbb R}}

\usepackage{colortbl}
\usepackage[dvipsnames]{xcolor}
\usepackage{xcolor}

\begin{document}
\setcounter{page}{1}

\title[Alternating iterations]
{Proper Weak Regular Splitting and
  its Application to Convergence of Alternating Iterations}

\author[Debasisha Mishra]{Debasisha Mishra$^*$}

\address{$^{*}$ Department of Mathematics, National Institute of Technology Raipur\\
                         Raipur - 492 010, India.}
\email{\textcolor[rgb]{0.00,0.00,0.84}{kapamath@gmail.com}}

%\dedicatory{This paper is dedicated to Professor Miroslav Fiedler}



\date{Received: xxxxxx; Revised: yyyyyy; Accepted: zzzzzz.
\newline \indent $^{*}$ Corresponding author}

\begin{abstract}

 Theory of matrix splittings is a useful tool for finding solution
of rectangular linear system of  equations, iteratively. The purpose
of this paper is  two-fold. Firstly, we revisit theory
 of   weak regular splittings for rectangular matrices. Secondly,  we propose
   an alternating iterative method for solving rectangular linear
    systems  by using  the Moore-Penrose inverse and discuss its convergence theory,
     by extending the work of
  Benzi and Szyld [Numererische Mathematik 76 (1997) 309-321;
  MR1452511]. Furthermore,  a comparison result is obtained which insures faster convergence rate
  of the proposed alternating iterative scheme.
\end{abstract} \maketitle

\section{Introduction}
Many questions in science and engineering give rise to linear
discrete ill-posed problems.  In particular, the discretization of
Fredholm integral equations of the first kind, and in particular
deconvolution problems with a smooth kernel, lead to linear systems
of equations of the form
\begin{eqnarray}\label{eq0}
Ax=b, ~~~~A\in {\R}^{m\times n}, ~~~~~x\in {\R}^n, ~~b\in {\R}^m,
\end{eqnarray}
 with a matrix of ill-determined rank, where ${\R}^{m\times n}$ denotes the set of all real rectangular matrices.
  Linear systems of equations
with a matrix of this kind are commonly referred to as linear
discrete ill-posed problems.  We consider equation (\ref{eq0}) as a
least-square problem in case  of a inconsistent system. Similarly,
singular linear  systems of equations arise in many problems like
finite difference representation of Neumann problems, finite element
electromagnetic analysis using edge elements and computation of
stationary probability vectors of stochastic matrices in the
analysis of Markov chains, to name a few. In particular, we arrive
at an $M$-matrix\footnote{A matrix $A\in {\R}^{n\times n}$ is said
to be an  {\it $M$-matrix} if $A = sI - B$, where $B\geq 0$ and
$s\geq \rho(B)$. It becomes a singular $M$-matrix when $s=\rho(B)$.}
as co-efficient matrix $A$.
 The study of $M$-matrices has a long history.   A systematic study of
$M$-matrices was first initiated by Fiedler and Ptak \cite{tak}.
Fifty equivalent conditions of an $M$-matrix are reported in the
book by Berman and Plemmons \cite{bpbook}. An extensive theory of
$M$-matrix has been developed for its role in iterative methods. In
both theoretical and practical case, iterative methods play a vital
role for solving large sparse linear system of equations as either
solvers or preconditioners. In this note, we propose an alternating
iterative method using theory of proper splittings.

For $A\in {\R}^{m\times n}$, a {\it splitting} is an expression of
the form
 $A = U-V$, where $U$ and $V$ are matrices of the same order as in $A$. The concept of
 splitting first arises from the iterative solution of large linear system of
equations. Standard iterative methods like the Jacobi, Gauss-Seidel
and successive over-relaxation methods arise from different choices
of $U$ and $V$.  Berman and Plemmons  \cite{bpcones} proposed first
an iterative method for solving (\ref{eq0}). They used the
Moore-Penrose inverse for computing least square solutions in the
inconsistent case.  The matrix $G$ satisfying the four matrix
equations: $AGA=A,~GAG=G,~(AG)^{T}=AG$ and $(GA)^{T}=GA$ is called
the {\it Moore-Penrose inverse} of $A$ (here $B^T $ denotes the
transpose of $B$). It always exists and unique, and is denoted by
$A^{\dag}$. The advantage of iterative technique for solving
rectangular system of linear equations is that it avoids use of the
normal system $A^TAx = A^Tb$ where $A^TA$ is frequently
ill-conditioned and influenced greatly by roundoff errors (see
\cite{f}).

Berman and Plemmons \cite{bpcones} introduced the notion of proper
splitting for rectangular matrices, which we recall next. A
splitting $A=U-V$ of $A\in {\R}^{m\times n}$ is called a {\it proper
splitting} if $R(U)=R(A)$ and $N(U)=N(A)$, where the text $R(A)$ and
$N(A)$ denote the range and null-space of a matrix $A$,
respectively.
 The authors of \cite{bpcones}  considered  the following iteration
 scheme:
\begin{eqnarray}\label{eq1}
x^{i+1}=Hx^{i}+c,
\end{eqnarray}
where   $A=U-V$ is a proper splitting, $H=U^{\dag}V \in
{\R}^{n\times n} $ is called the {\it iteration matrix}  and
$c=U^{\dag}b$ to solve  (\ref{eq0}), iteratively. The same authors
proved that
 the iteration scheme (\ref{eq1})
converges
 to $A^{\dag}b$, the least square solution of minimum norm
 for any initial vector $x^{0}$ if and only if
 the spectral radius of $H$ is less than 1
  (see Corollary 1, \cite{bpcones}).

The authors of \cite{bpcones} also obtained several convergence
criteria for different subclasses of proper splitting. Recently,
Jena {\it et al.} \cite{miscal} revisited the same theory. Certain
necessary parts of the same theory are recalled and discussed in
Section 3 of this paper. The above discussion extends convergence
theory of the iterative scheme:
\begin{eqnarray}\label{eq1non}
x^{i+1}=U^{-1}Vx^{i}+U^{-1}b,
\end{eqnarray}
which is being used to solve  square nonsingular linear system
$Ax=b$.

On the other hand, the speed of the iteration schemes (\ref{eq1})
 and (\ref{eq1non}) is a subject of concern.
In this direction, several works have been done in literature. Among
these works, Benzi and Szyld \cite{benz} proposed the concept of
alternating iteration method for solving square nonsingular and
square singular linear system of the form $Ax=b$, iteratively. They
considered two splittings of $A\in {\R}^{n\times n}$ such that
$A=M-N=P-Q$, and proposed the scheme
\begin{eqnarray}\label{AIS0} x^{i+1/2}=M^{-1}Nx^{i}+M^{-1}b,
~~~~x^{i+1}=P^{-1}Qx^{i+1/2}+P^{-1}b, ~i=0,1,2,\cdots.
\end{eqnarray}
Then, eliminating $x^{i+1/2}$, they obtained
\begin{eqnarray}\label{AIS}
~~~~x^{i+1}=P^{-1}QM^{-1}N x^{i}+P^{-1}(QM^{-1}+I)b,
~~~i=0,1,2,\cdots.
\end{eqnarray}
Finally, they discussed   convergence theory of the above scheme
using weak regular splitting of $A$ among other results. (Recall
that a splitting $A=U-V$ of $A\in {\R}^{n\times n}$ is weak regular
\cite{var} if $U^{-1}$ exists,  $U^{-1}\geq 0$ and $U^{-1}V\geq 0$.)
 {\it The objective of the present
paper is to introduce alternating iteration technique and to develop
its  convergence theory for solving rectangular
 linear system of equations.} By doing this, we will have another iteration scheme  of the
 form (\ref{eq3})
 which converges faster than  the iteration scheme (\ref{eq1}).

 To fulfil
this objective, we organize the content of the paper as follows. In
Section 2, we set up our notation and terminology. Furthermore,
  we collect some useful facts on projection, the Moore-Penrose inverse, proper splittings,
 spectral radius and its connection with
  non-negative matrices which will be used in deriving the main results in Section 3 and Section 4.
 The next Section  recalls results on theory of regular and weak
regular splittings for rectangular matrices. It also contains two
comparison results which will help us in detecting a better
splitting between  matrix splittings.
 The main contribution of this paper discussed in   Section 4 is that we
introduce the notion of alternating iterative scheme for rectangular
matrices by using the Moore-Penrose inverse. Then convergence and
comparison results involving this scheme are reported. Finally, we
end up with a concluding Section which compares our work with  Benzi
and Szyld's work.

\section{Prerequisites}
This section contains our notation and definitions, and also we
recall some useful facts related to Perron-Frobenous theory for
non-negative matrices. Throughout the paper, all our matrices are
real.
 Let  $L$ and $M$ be   complementary subspaces of ${\R}^{n}$,
i.e., $L\oplus M={\R}^{n}$. Let also  $P_{L,M}$  be a  projector on
$L$ along $M$. Then $P_{L,M}A=A$ if and only if $R(A)\subseteq L$
and $AP_{L,M}=A$ if and only if $N(A)\supseteq M$. If $L\perp M$,
then $P_{L,M}$  will be denoted by $P_{L}$.
 The {\it spectral radius} of  $A\in {\R}^{n\times n}$, denoted by $\rho(A)$ is  defined by
$\rho(A)=\displaystyle \max_{1\leq i\leq n}|\lambda_i|$, where
$\lambda_1,\lambda_2,\cdots,\lambda_n$  are the eigenvalues of $A$.
It is known that $\rho(AB)=\rho(BA)$, where $A$  and $B$ are two
matrices such that $AB$ and $BA$ are defined.   We now recall some
facts on generalized inverses, non-negative matrices and proper
splittings below.

\subsection{Generalized inverses} These are generalizations of the
ordinary matrix inverse. Generalized inverses exist for all matrices
while the ordinary matrix inverse does not exist. Some of the
important generalized inverses are the Moore-Penrose inverse, the
group inverse and the Drazin inverse. While the definition of the
first one is introduced in  page 2, the other two are presented
next.
    The {\it Drazin inverse} of a matrix
$A\in {\R}^{n\times n}$
 is the unique solution $X\in {\R}^{n\times n}$ satisfying  the equations:
$A^k=A^kXA$, $X=XAX$ and $AX=XA$, where $k$ is the
index\footnote{The {\it index} of $A\in {\R}^{n\times n}$ is the
least non-negative integer $k$ such that
rank($A^{k+1}$)=rank($A^{k}$).} of $A$. It is denoted by $A^{D}$.
But for $k=1$, $A^{D}$ is called as {\it group inverse} of $A$, and
is denoted by $A^{\#}$. While $A^{\dag}$ and $A^{D}$ exist for any
matrix $A$, $A^{\#}$ does not. It exists only for matrices of index
1. We refer to \cite{ben} for more details. In case of nonsingular
matrix $A$, $A^{\dag}=A^{-1}=A^D=A^{\#}$. Some of the well-known
properties of $A^{\dag}$   which will be
 frequently used in this paper are: $R(A^{T})=R(A^{\dag})$;
$N(A^{T})=N(A^{\dag})$;
 $AA^{\dag}=P_{R(A)};~A^{\dag}A=P_{R(A^{T})}$.
  In particular, if $x\in R(A^{T})$ then $x=A^{\dag}Ax$.

\subsection{Non-negative matrices}
 $A=(a_{ij}) \in {\R}^{m\times n}$ is called {\it non-negative} if $A\geq 0$,
 where $A\geq 0$  means   $a_{ij} \geq 0$ for each $i, j$, and there exists
at least one pair of indices  $k, l$ for which $a_{k,l}> 0.$ For
$A,B\in {\R}^{m\times n}$, $A\leq B$ means $B-A\geq 0$.  Similarly,
$B>0$ means all the entries of $B$ are positive. The same notation
and nomenclature are also used for vectors. A matrix
 $A \in {\R}^{m\times n}$ is called {\it semi-monotone} if $A^{\dag}\geq
 0$.   Next four results deal with non-negativity and  spectral
 radius,
and are going to be used in  Section 3 and Section 4.

\begin{thm}\label{frob}(Theorem 2.20, \cite{var})\\
Let  $B\in {\R}^{n\times n}$ and  $B\geq 0$. Then \\
(i) $B$ has a non-negative real eigenvalue equal to its spectral radius.\\
(ii) There exists a non-negative eigenvector for its spectral
radius.
\end{thm}

\begin{thm}\label{lspecageb}(Theorem 2.21, \cite{var})\\
Let $A, ~B\in {\R}^{n\times n}$  and $A\geq  B \geq 0.$ Then
$\rho(A)\geq \rho(B).$
\end{thm}

\begin{thm}\label{neuman}(Theorem 3.15, \cite{var})\\
Let $X\in {\R}^{n\times n}$  and $X\geq 0$. Then $\rho(X)<1$ if and
only if $(I-X)^{-1}$ exists and
$(I-X)^{-1}=\displaystyle\sum_{k=0}^{\infty}X^{k}\geq 0$.
\end{thm}

\begin{thm} \label{use} (Theorem 1.11, \cite{bpbook}, Chapter 2)\\
Let $B\in {\R}^{n\times n}$, $B\geq  0$ and  $x> 0$  be such that
$Bx- \alpha x \leq 0$. Then $\rho(B) \leq  \alpha $.
\end{thm}

\subsection{Proper splittings}
 Here, we recall some  results on  proper splittings
   which are useful in proving our main results.
  The first one contains a
  few properties of a proper splitting.

\begin{thm}\label{proper} (Theorem 1, \cite{bpcones})\\
Let  $A=U-V$ be a proper splitting of $A\in {\R}^{m\times n}$. Then \\
(a) $A=U(I-U^{\dag}V)$;\\
(b) $I-U^{\dag}V$ is invertible;\\
(c) $A^{\dag}= (I-U^{\dag}V)^{-1}U^{\dag}.$
\end{thm}

If $A=U-V$ is a proper splitting of $A\in {\R}^{m\times n}$, then
$U=A+V$ is also a proper splitting. Thus $I+A^{\dag}V$ is invertible
by Theorem \ref{proper} (b). Since $FG$ and $GF$ have same
eigenvalues for any $F$ and $G$ such that both the product are
defined, and $I+A^{\dag}V$ is invertible, so $-1$ is not an
eigenvalue of $VA^{\dag}$. Hence $I+VA^{\dag}$ is invertible. This
fact can also be proved by considering the proper splitting $U^T=
A^T+V^T$.

The next lemma shows a relation between the eigenvalues of
$U^{\dag}V$ and $A^{\dag}V$.

\begin{lem}\label{lemqeigval}(Lemma 2.6, \cite{misaml})\\
Let  $A=U-V$ be a  proper splitting of $A\in {\R}^{m\times n}$. Let
$\mu_i, ~1\leq i \leq s$ and $\lambda_j, ~1\leq j \leq s$ be the
 eigenvalues of the matrices $U^{\dag}V$ and $A^{\dag}V$,
respectively. Then for every $j$, we have $1 + \lambda_j\neq 0$.
Also, for every $i$, there exists $j$ such that
$\mu_i=\frac{\lambda_j}{1+\lambda_j}$ and for every $j$, there
exists $i$ such that $\lambda_j =\frac{\mu_i}{1-\mu_i}.$
\end{lem}

\section{Proper Regular $\&$ Proper Weak Regular Splittings}

In this section,   the theory of proper regular and weak regular
splittings is recalled first, and then some new results are
proposed. We reproduce the definitions of  proper regular splitting
and proper weak regular splitting below.

\begin{defn}\label{r}(Definition 2, \cite{cli} $\&$ Definition 1.1, \cite{miscal})
 A splitting $A=U-V$ of $A\in {\R}^{m\times n}$ is called a {\it proper regular
 splitting}
 if it is a proper
splitting such that $U^{\dag}\geq 0$ and $V\geq 0$.
\end{defn}

\begin{defn}\label{wr} (Definition 1.2, \cite{miscal})
 A splitting $A=U-V$ of $A\in {\R}^{m\times n}$  is called a  {\it proper weak regular
 splitting}
 if it is a proper
splitting such that $U^{\dag}\geq 0$ and $U^{\dag}V\geq 0$.
\end{defn}

The class of matrices having a fixed positive real number in all the
entries   always have proper regular and proper weak regular
splittings. We next present an example of a proper splitting which
is a proper weak regular splitting but not a proper  regular
splitting.

\begin{ex}
Let $A= \begin{bmatrix}
                                                 \begin{array}{ccc}
                                                        9 & -8 & 15\\
                                                         -6 & 6 & -10\\
                                                        \end{array}
                                                        \\
                                                        \end{bmatrix} = \begin{bmatrix}
                                                 \begin{array}{ccc}
                                                        6 & -4 & 10\\
                                                         -3 & 4 & -5\\
                                                        \end{array}
                                                        \\
                                                       \end{bmatrix}- \begin{bmatrix}
                                                 \begin{array}{ccc}
                                                        -3 & 4 & -5\\
                                                         3 & -2 & 5\\
                                                        \end{array}
                                                        \\
                                                        \end{bmatrix}\\ =U-V.$ Then
                                                        $R(U)=R(A)$,
                                                        $N(U)=N(A)$,
 $U^{\dag}=\begin{bmatrix}
                                                 \begin{array}{ccc}
                                                        3/34 & 3/34 \\
                                                        1/4 & 1/2 \\
                                                          5/34 & 5/34 \\
                                                        \end{array}
                                                        \\
                                                        \end{bmatrix}\geq 0$
                                                         and $U^{\dag}V=\begin{bmatrix}
                                                 \begin{array}{ccc}
                                                        0 & 3/17 & 0\\
                                                        3/4 & 0 & 5/4\\
                                                          0 & 5/17 & 0\\
                                                        \end{array}
                                                        \\
                                                        \end{bmatrix}\geq 0$. Thus $A=U-V$
                                                        is a
                                                          proper weak
regular splitting but not a proper regular splitting since $V\ngeq
0$.
\end{ex}

Berman and Plemmons \cite{bpcones} initiated the study of
convergence theory of iteration scheme (\ref{eq1})  without terming
the class of proper splittings $A=U-V$ as  the proper regular and
proper weak regular splittings. Two of their results presented below
characterize semi-monotone matrices in terms of these class of
splittings.

\begin{thm}\label{PRC} ( Theorem 1.3, \cite{miscal})
 Let $A = U-V$ be a proper regular splitting of $A\in {\R}^{m\times n}$. Then
 $A^{\dag}\geq 0$ if and only if $\rho(U^{\dag}V) < 1.$
\end{thm}

\begin{thm}   (Theorem 3, \cite{bpcones}) \label{PWRC}
Let $A = U-V$ be a proper weak regular splitting of $A\in
{\R}^{m\times n}$. Then
 $A^{\dag}\geq 0$ if and only if $\rho(U^{\dag}V) < 1.$
\end{thm}

Noted next  result  is proved in \cite{miscal} which contains
equivalent convergence  condition for iteration scheme (\ref{eq1}).

\begin{thm}\label{PRC2}(Theorem 3.1, \cite{miscal})
Let  $A=U-V$ be a    proper regular splitting of $A\in {\R}^{m\times
n}$. If
 $A^{\dag}\geq 0$, then\\
(a)  $A^{\dag}\geq U^{\dag}$;\\
(b) $\rho(A^{\dag}V)\geq \rho(U^{\dag}V)$;\\
(c) $\rho(U^{\dag}V)=\frac{\rho(A^{\dag}V)}{1+\rho(A^{\dag}V)}<1$.
 \end{thm}

The conditions of the proper weak regular splitting still can be
weakened by dropping  the condition $U^{\dag}\geq 0$, and the
resultant splitting is known as  {\it proper nonnegative (proper
weak) splitting} (Definition 3.1, \cite{miscam}). A convergence
result for a proper nonnegative splitting is obtained below.

\begin{lem} \label{cam1}(Lemma 3.4, \cite{miscam})
Let $A=U-V$ be a proper nonnegative  splitting of $A\in
{\R}^{m\times n}$ and  $A^{\dag}U\geq 0$.
 Then $\rho(U^{\dag}V)=\frac{\rho(A^{\dag}U)-1}{\rho(A^{\dag}U)}<1$.
\end{lem}

We remark that the above result is also true for the proper weak
regular splitting. Next result further adds a few more equivalent
conditions to the above Lemma for a proper weak regular splitting.

 \begin{thm}\label{PWRC2}
 Let    $A=U-V $ be a proper weak regular splitting of $ A\in {\R}^{m\times n}$.
  Then $(a)\Rightarrow (b) \Rightarrow (c)\Rightarrow (d)\Rightarrow
   (e)\Rightarrow (f)\Rightarrow (g)$.\\
(a) $A^{\dag}U\geq 0$;\\
(b) $\rho(U^{\dag}V)=\frac{\rho(A^{\dag}U)-1}{\rho(A^{\dag}U)}$;\\
(c) $\rho(U^{\dag}V)<1$;\\
(d) $(I-U^{\dag}V)^{-1}\geq 0$;\\
(e) $A^{\dag}V\geq 0$;\\
(f) $A^{\dag}V\geq U^{\dag}V$;\\
(g) $\rho(U^{\dag}V)=\frac{\rho(A^{\dag}V)}{1+\rho(A^{\dag}V)}<1$.
\end{thm}

\begin{proof}
$ (a)\Rightarrow (b)$: Follows from the proof of Lemma 3.7.\\
$ (b)\Rightarrow (c)$: Obvious.\\
 $ (c)\Rightarrow (d)$: The conditions
$\rho(U^{\dag}V)<1$ and  $U^{\dag}V\geq 0$ together yields that
$(I-U^{\dag}V)^{-1}=\displaystyle \sum_{k=0}^{\infty}(U^{\dag}V)^{k}
\geq 0$, by
Theorem \ref{neuman}.\\
$ (d)\Rightarrow (e)$:  By Theorem \ref{proper} (c), we obtain
$A^{\dag}=(I-U^{\dag}V)^{-1}U^{\dag}$. Post-multiplying $V$ both the
sides, we get $A^{\dag}V=(I-U^{\dag}V)^{-1}U^{\dag}V$. Hence
$A^{\dag}V\geq 0$ as $(I-U^{\dag}V)^{-1}\geq 0$
and  $U^{\dag}V\geq 0$.\\
$ (e)\Rightarrow (f)$: We have
$A^{\dag}V=(I-U^{\dag}V)^{-1}U^{\dag}V$ by Theorem \ref{proper} (c).
Pre-multiplying $I-U^{\dag}V$ both the sides, we obtain
 $(I-U^{\dag}V)A^{\dag}V=U^{\dag}V$ which implies
$A^{\dag}V-U^{\dag}V=U^{\dag}VA^{\dag}V$.  Thus
 $A^{\dag}V\geq U^{\dag}V$ as $U^{\dag}V\geq 0$ and
  $A^{\dag}V\geq 0$.\\
$(f)\Rightarrow (g)$: Observe that $A^{\dag}V\geq 0$ as
 $U^{\dag}V\geq 0$.
  Let $\lambda$  be any eigenvalue of
$A^{\dag}V$ and $f(\eta)=\frac{\eta}{1+\eta}, ~~\eta \geq 0$. Then
$f$ is a strictly increasing function. Let $\mu$ be
 any eigenvalue of
 $U^{\dag}V$.
 We now have  $\mu=\frac{\lambda}{1+\lambda}$ by Lemma \ref{lemqeigval}.
  Hence,  $\mu$  attains its
  maximum
when $\lambda$ is maximum. But $\lambda$ is maximum when
$\lambda=\rho(A^{\dag}V)$.  As a result, the maximum value of $\mu$
is $\rho(U^{\dag}V)$. Thus
$\rho(U^{\dag}V)=\frac{\rho(A^{\dag}V)}{1+\rho(A^{\dag}V)}<1$.
\end{proof}

 The rate of convergence of the iteration scheme
(\ref{eq1}) depends on the spectral radius of the iteration matrix
$U^{\dag}V$. Hence, the spectral radius of the iteration matrix
plays a vital role in comparison of the speed of convergence of
different iterative schemes of the same linear system given in
(\ref{eq0}). Next result compares spectral radii of the iteration
matrices between a proper regular splitting and a proper weak
regular splitting arising out of the
 same coefficient matrix $A$.

\begin{thm}\label{tcomp1} Let $A=B-C$ be a proper weak regular
splitting and $A=U-V$ be a proper regular splitting   of a
semi-monotone matrix $A\in {\R}^{m\times n}$. If $A\geq 0$ and
$B^{\dag}\geq U^{\dag}$, then
$$ \rho(B^{\dag}C)\leq \rho(U^{\dag}V)<1 .$$
\end{thm}

\begin{proof}
By Theorem \ref{PRC} and Theorem \ref{PWRC},  we have
$\rho(U^{\dag}V)<1$ and $\rho(B^{\dag}C)< 1$.  Also
$\rho(U^{\dag}V)$ and $\rho(B^{\dag}C)$ are strictly increasing
functions of $\rho(A^{\dag}V)$ and $\rho(A^{\dag}C)$, so it suffices
to show that
$$\rho(A^{\dag}V)\geq \rho(A^{\dag}C).$$
But    $I+A^{\dag}C$ and $I+VA^{\dag}$ are both invertible as
$A=B-C=U-V$ are proper splittings. The conditions $A=B-C$
  is a proper weak regular splitting and $\rho(B^{\dag}C)<1$ implies that $A^{\dag}C\geq
  0$ by Theorem \ref{neuman} and Theorem \ref{proper} (c) which in turn
   yields $I+A^{\dag}C\geq  0$. Clearly, $I+VA^{\dag}\geq
  0$. Now $B^{\dag}\geq U^{\dag}$ yields
$A^{\dag}(I+VA^{\dag})\geq (I+A^{\dag}C)A^{\dag}$ i.e.,
$A^{\dag}VA^{\dag}\geq A^{\dag}CA^{\dag}.$ Then, post-multiplying it
by $V$, we have
$$(A^{\dag}V)^{2}\geq A^{\dag}CA^{\dag}V.$$ Again, post-multiplying $A^{\dag}VA^{\dag}\geq A^{\dag}CA^{\dag}$ by $A$,
we get $A^{\dag}VA^{\dag}A=A^{\dag}V \geq A^{\dag}CA^{\dag}A=
A^{\dag}C$. So
$$A^{\dag}VA^{\dag}C\geq (A^{\dag}C)^{2}.$$ Therefore, by Theorem
\ref{lspecageb}, we have
$$\rho^2(A^{\dag}V)\geq \rho(A^{\dag}VA^{\dag}C)= \rho(A^{\dag}CA^{\dag}V)
 \geq \rho^2(A^{\dag}C).$$ Hence
$\rho(A^{\dag}V)\geq \rho(A^{\dag}C).$ Thus
$$ \rho(B^{\dag}C)\leq \rho(U^{\dag}V)<1 .$$
\end{proof}

%A new class of splittings  is introduced next to avoid the condition
%$A\geq 0$ in the above theorem.
%
%\begin{defn} A splitting $A=U-V$ of $A\in {\R}^{m\times n}$ is called a
% {\it strictly proper regular
% splitting}
% if it is a proper
%splitting such that $U^{\dag}> 0$ and $V\geq 0$.
%\end{defn}
%
%We now present a result which does not need the condition $A\geq 0$
%but compares convergence rate of a strictly proper regular and a
%proper weak regular splitting of $A$.

We now present a result which replaces the condition $A\geq 0$ in
the above theorem by row sums of $U^{\dag}$  are positive.

%\begin{thm}\label{tcomp2}
%Let $A=B-C$ be a proper weak regular splitting and $A=U-V$ be a
%strictly proper regular   of a semi-monotone matrix $A\in
%{\R}^{m\times n}$. If $B^{\dag}\geq U^{\dag}$, then
%$$ \rho(B^{\dag}C)\leq \rho(U^{\dag}V)<1 .$$
%\end{thm}

\begin{thm}\label{tcomp2}
Let $A=B-C$ be a proper weak regular splitting and $A=U-V$ be a
 proper regular   splitting of a semi-monotone matrix $A\in
{\R}^{m\times n}$. If $B^{\dag}\geq U^{\dag}$ and row sums of
$U^{\dag}$  are positive, then
$$ \rho(B^{\dag}C)\leq \rho(U^{\dag}V)<1 .$$
\end{thm}

\begin{proof}
We have $ \rho(U^{\dag}V)<1$ and $\rho(B^{\dag}C)<1$,  by Theorem
\ref{PRC} and  Theorem \ref{PWRC}, respectively. As $U^{\dag}V\geq
0$, by Theorem \ref{frob}, there exists $x\geq 0$ such that
$U^{\dag}Vx=\rho(U^{\dag}V) x$. So $x\in R(U^T)=R(B^T)$. Therefore
$Ux=\frac{1}{\rho(U^{\dag}V)}UU^{\dag}Vx=\frac{1}{\rho(U^{\dag}V)}Vx$.
Now
$Ax=(U-V)x=U(I-U^{\dag}V)x=(1-\rho(U^{\dag}V))Ux=(\frac{1}{\rho(U^{\dag}V)}-1)Vx\geq
0$ as $V\geq 0$ and $\rho(U^{\dag}V) < 1$. Then the condition
$B^{\dag}\geq U^{\dag}$ yields $B^{\dag}Ax\geq U^{\dag}Ax$, i.e.,
$B^{\dag}(B-C)x \geq U^{\dag}(U-V)x$ which in turn implies that
$x-B^{\dag}Cx\geq x-U^{\dag}Vx$. Hence $B^{\dag}Cx\leq
U^{\dag}Vx=\rho(U^{\dag}V)x$. By replacing $A$ by $A-\epsilon J$ and
$V$ by $V+\epsilon J$, where all the entries of $J$ are 1, and
$\epsilon$ is a small positive real number, we can assume that
$x>0$. Thus $ \rho(B^{\dag}C)\leq \rho(U^{\dag}V)<1,$ by Theorem
\ref{use}.
\end{proof}

We remark that the above result is also true if we replace the
condition `row sums of $U^{\dag}$ are positive' by `no row of
$U^{\dag}$ is zero' as  the conditions `$U^{\dag}\geq 0$' and `no
row of $U^{\dag}$ is zero' yield `row sums of $U_{2}^{\dag}$ are
positive'.  The above proof adopts a similar technique as in the
proof of Lemma (Section 3, \cite{els}). Notice that
$U^{\dag}(V+\epsilon J) > 0$ may not be possible always unless row
sums of $U^{\dag}$  are positive. Hence we have assumed the
condition row sums of $U^{\dag}$ are positive. This fact is shown
through an example below.

\begin{ex}
Let $A= \begin{bmatrix}
                                                 \begin{array}{ccc}
                                                        0 & 2 & 1\\
                                                         0& 4 & 2\\
                                                        \end{array}
                                                        \\
                                                        \end{bmatrix}=\begin{bmatrix}
                                                 \begin{array}{ccc}
                                                        0 & 4 & 2\\
                                                         0 & 8  & 4\\
                                                        \end{array}
                                                        \\
                                                        \end{bmatrix}-\begin{bmatrix}
                                                 \begin{array}{ccc}
                                                        0 & 2 & 1\\
                                                         0 & 4  & 2\\
                                                        \end{array}
                                                        \\
                                                        \end{bmatrix}$.
                                                        We have
                                                        $R(U)=R(A)$,
                                                        $N(U)=N(A)$,
                                                        $V\geq 0$
                                                        and $U^{\dag}=\begin{bmatrix}
                                                 \begin{array}{cc}
                                                         0 & 0\\
                                                         1/25 & 2/25\\
                                                        1/50 & 1/25\\
                                                        \end{array}
                                                        \end{bmatrix}\geq
                                                        0$.
                                                         Hence $A=U-V$
                                                        is a
                                                        proper
                                                        regular
                                                        splitting.  But for $\epsilon = 0.01$,
                                                         we have $U^{\dag}(V+\epsilon J)=\begin{bmatrix}
                                                 \begin{array}{ccc}
                                                        0 & 0 & 0\\
                                                        3/2500 & 1003/2500 & 503/2500\\
                                                        3/5000 & 535/2667 & 218/2167\\
                                                        \end{array}
                                                        \end{bmatrix}\geq
                                                        0$.
\end{ex}

One can use the comparison results to pick the best splitting among
any finite number of splittings. However, the major drawback of this
theory is the following: {\it it is time consuming and needs many
computation.} To avoid this situation and to get a finer process, we
now proceed to introduce  alternating iteration scheme for
rectangular matrices replacing the ordinary matrix inverse by
Moore-Penrose inverse, and then discuss its convergence theory.

\section{Application to Convergence of Alternating  Iterations}
Let $A=M-N=U-V$ be two proper splittings of $A\in {\R}^{m \times
n}$. We now propose
\begin{eqnarray}\label{eq2} x^{i+1/2}=M^{\dag}Nx^{i}+M^{\dag}b,
~~~~x^{i+1}=U^{\dag}Vx^{i+1/2}+U^{\dag}b, ~~~i=0,1,2,\cdots,
\end{eqnarray}
 as the general class of iterative method for finding the solution of
 (\ref{eq0}) with the initial approximation $x^0$. In case
 nonsingular $M$ and $U$, the above equation reduces to equation (8)
 of section 3, \cite{benz} (i.e., equation (\ref{AIS0}) of this paper). Not only that many well-known methods
 belong to such a class, and are also discussed in the same section of  \cite{benz}.

 In order to study convergence of the above scheme, we construct a
 single splitting $A=B-C$ associated with the iteration matrix by
 eliminating $x^{i+1/2}$ from (\ref{eq2}). So, we have
\begin{eqnarray}\label{eq3}
~~~~x^{i+1}=U^{\dag}V M^{\dag}N x^{i}+U^{\dag}(VM^{\dag}+I)b,
~~~i=0,1,2,\cdots,
\end{eqnarray}
where $H=U^{\dag}V M^{\dag}N$ is the iteration matrix of the new
iterative scheme (\ref{eq3}).

Recall that the convergence of the individual splittings $A=M-N$ and
$A=U-V$ does not imply the convergence of the alternating iterative
scheme (\ref{eq3}). Example 3.1, \cite{benz} is in this direction,
and  is obtained below for the sack of completeness and ready
reference.

\begin{ex}(Example 3.1, \cite{benz})\\
Let $A= \begin{bmatrix}
                                                 \begin{array}{cc}
                                                        2 & -1\\
                                                         -1 & 2\\
                                                        \end{array}
                                                        \\
                                                        \end{bmatrix}$,
                                                        $M= \begin{bmatrix}
                                                 \begin{array}{cc}
                                                        2 & 1\\
                                                         -1 & 1\\
                                                        \end{array}
                                                        \\
                                                        \end{bmatrix}$
                                                        and $U= \begin{bmatrix}
                                                 \begin{array}{cc}
                                                        1 & -1\\
                                                         1 & 2\\
                                                        \end{array}
                                                        \\
                                                        \end{bmatrix}$.
                                                        Then
                                                        $A=M-N=U-V$
                                                        are two
                                                        convergent
                                                        proper
                                                        splittings,
                                                        but  $\rho(H)=\rho(U^{\dag}V M^{\dag}N)=1.$
\end{ex}

Convergence of the  iteration scheme  (\ref{eq3}) is addressed in
the next result.

\begin{thm}\label{Main13}
Let $A=M-N=U-V$ be two proper weak regular splittings of a
semi-monotone matrix $A\in {\R}^{m\times n}$. Then $\rho
(H)=\rho(U^{\dag}V M^{\dag}N)<1.$
\end{thm}

\begin{proof}
We have $H=U^{\dag}V M^{\dag}N=U^{\dag}(U-A) M^{\dag}(M-A)=
U^{\dag}U-U^{\dag}A-M^{\dag}A+U^{\dag}AM^{\dag}A.$ Since $A=M-N=U-V$
are two proper splittings, so $R(U)=R(M)=R(A)$ and $N(U)=N(M)=N(A)$.
Hence $M^{\dag}M=U^{\dag}U=A^{\dag}A$. We then have
$H=U^{\dag}U-U^{\dag}A-M^{\dag}A+U^{\dag}AM^{\dag}A$. Again,
$U^{\dag}AM^{\dag}=U^{\dag}(U-V)M^{\dag}=U^{\dag}UM^{\dag}-U^{\dag}VM^{\dag}=
M^{\dag}MM^{\dag}-U^{\dag}VM^{\dag}=M^{\dag}-U^{\dag}VM^{\dag}.$ But
$U^{\dag}VM^{\dag}\geq 0$ as $A=M-N=U-V$ are two proper weak regular
splittings. So $M^{\dag}\geq M^{\dag}-U^{\dag}VM^{\dag} \geq
U^{\dag}AM^{\dag}$ which results $(I-U^{\dag}A)M^{\dag} \geq 0$. The
condition $U^{\dag}\geq 0$ yields $U^{\dag}+(I-U^{\dag}A)M^{\dag}=
U^{\dag}+M^{\dag}-U^{\dag}AM^{\dag} \geq 0$. This  implies
$A^{\dag}-U^{\dag}-M^{\dag}+U^{\dag}AM^{\dag}\leq A^{\dag}$ which
 can be rewritten as
$U^{\dag}UA^{\dag}-U^{\dag}AA^{\dag}-M^{\dag}AA^{\dag}+U^{\dag}AM^{\dag}AA^{\dag}\leq
A^{\dag}$. We then have
$(U^{\dag}U-U^{\dag}A-M^{\dag}A+U^{\dag}AM^{\dag}A)A^{\dag}\leq
A^{\dag}$, i.e., $HA^{\dag}\leq A^{\dag}$. Thus $(I-H)A^{\dag}\geq
0$.

As $H\geq 0$, we  have $0\leq (I+H+H^2+H^3+\cdots
+H^m)(I-H)A^{\dag}= (I-H^{m+1})A^{\dag}\leq A^{\dag}$ for each $m\in
\mathbb N$. So, the partial sums of the series $ \displaystyle
\sum_{m=0}^{\infty} H^m$ is uniformly bounded. Hence $\rho (H)<1$.
\end{proof}

Next example shows that the converse of Theorem \ref{Main13} is not
true.

\begin{ex}
Let $A= \begin{bmatrix}
                                                 \begin{array}{ccc}
                                                        1 & 0 & 1\\
                                                         0& 1 & 1\\
                                                        \end{array}
                                                        \\
                                                        \end{bmatrix}$,
                                                         $M=\begin{bmatrix}
                                                 \begin{array}{ccc}
                                                        4 & 0 & 4\\
                                                         2 & 2  & 4\\
                                                        \end{array}
                                                        \\
                                                        \end{bmatrix}$
                                                        and $U=\begin{bmatrix}
                                                 \begin{array}{ccc}
                                                        2 & 0 & 2\\
                                                         1 & 2  & 3\\
                                                        \end{array}
                                                        \\
                                                        \end{bmatrix}$.
                                                        Then
                                                        $A=M-N=U-V$
                                                        are two
                                                        proper
                                                        splittings.
                                                        Also $\rho (H)=\rho(U^{\dag}V
M^{\dag}N)=3/8<1$. But $M^{\dag}=\begin{bmatrix}
                                                 \begin{array}{cc}
                                                        1/4 & -1/6\\
                                                         -1/4 & 1/3\\
                                                        0 & 1/6
                                                        \end{array}
                                                        \end{bmatrix}\ngeq 0$  and $U^{\dag}=\begin{bmatrix}
                                                 \begin{array}{cc}
                                                        5/12 & -1/6\\
                                                         -1/3 & 1/3\\
                                                        1/12 & 1/6
                                                        \end{array}
                                                        \end{bmatrix}\ngeq 0$, i.e.,  $A=M-N=U-V$ are not proper weak regular splittings.
\end{ex}

It is of interest to know  the type of splitting $B-C$ of $A$ that
yields the iterative scheme (\ref{eq3})(i.e.,
$x^{i+1}=Hx^{i}+B^{\dag}b$ with $H=B^{\dag}C$). This can be restated
as what can we say about the type of the induced splitting $A=B-C$
which is induced by $H=U^{\dag}V M^{\dag}N$. The same problem is
settled partially by the next result under the assumptions of a few
conditions.

\begin{thm}\label{Main23}
Let $A=M-N=U-V$ be two proper weak regular splittings of a
semi-monotone matrix $A\in {\R}^{m\times n}$. Then the unique
splitting $A=B-C$ induced by $H$ with $B=M(M+U-A)^{\dag}U$ is a
proper weak regular splitting if $R(M+U-A)=R(A)$ and
$N(M+U-A)=N(A)$.
\end{thm}

\begin{proof}
From equation (\ref{eq3}), we have $B^{\dag}=U^{\dag}(VM^{\dag}+I)$.
By substituting $V=U-A$, we get
$B^{\dag}=U^{\dag}+U^{\dag}UM^{\dag}-U^{\dag}AM^{\dag}=
U^{\dag}MM^{\dag}+U^{\dag}UM^{\dag}-U^{\dag}AM^{\dag}=U^{\dag}(M+U-A)M^{\dag}$.
Since $R(M+U-A)=R(A)$, $N(M+U-A)=N(A)$ and $A=M-N=U-V$ are proper
splittings,  we have
$(M+U-A)(M+U-A)^{\dag}=P_{R(M+U-A)}=P_{R(A)}=P_{R(U)}=P_{R(M)}$ and
$(M+U-A)^{\dag}(M+U-A)=P_{R((M+U-A)^T)}=P_{R(A^T)}=P_{R(M^T)}=P_{R(U^T)}.$
Let $X=M(M+U-A)^{\dag}U$, then $B^{\dag}X=
U^{\dag}(M+U-A)M^{\dag}M(M+U-A)^{\dag}U=U^{\dag}P_{R(U)}U=U^{\dag}U.$
So $B^{\dag}X$ is symmetric and $B^{\dag}X B^{\dag}=B^{\dag}$.
Similarly, it can be shown that $XB^{\dag}$ is symmetric and
$XB^{\dag}X=X$. Hence $X=(B^{\dag})^{\dag}=B=M(M+U-A)^{\dag}U$.

Next to show that $R(B)=R(A)$ and $N(B)=N(A)$.  First, we prove that
$N(U)=N(A)=N(B)$. Clearly, $N(U)\subseteq N(B)$. Let $Bx=0$.
Pre-multiplying $M^{\dag}$ to $Bx=0$ and using
$M^{\dag}M=P_{R(A^T)}=P_{R((M+U-A)^T)}$, we obtain
$(M+U-A)^{\dag}Ux=0.$ Again, pre-multiplying $(M+U-A)$ and using the
fact $(M+U-A)(M+U-A)^{\dag}=P_{R(M+U-A)}=P_{R(U)}$, we get $x\in
N(U)$. So $N(B)\subseteq N(U)$.  We   next have to prove that
$R(A)=R(B)$, i.e., $N(M^T)=N(A^T)=N(B^T)$. Since
$B=M(M+U-A)^{\dag}U$, so $N(M^T)\subseteq N(B^T)$. Hence we need to
show the other way, i.e., $N(B^T)\subseteq N(M^T)$. Let $x \in
N(B^T)$. Then $(M(M+U-A)^{\dag}U)^T x=0$. Pre-multiplying
$(U^{\dag})^T$, we get $(UU^{\dag})^T [(M+U-A)^{\dag} ]^T M^T x=0$,
i.e., $x^T M (M+U-A)^{\dag}  UU^{\dag}=x^T M (M+U-A)^{\dag}=0$.
Again, post-multiplying $(M+U-A)$, we get $x^T MM^{\dag}M=0$. Thus
 $M^T x=0$, i.e., $N(B^T)\subseteq N(M^T)$.

We have
$B^{\dag}=M^{\dag}+U^{\dag}-U^{\dag}AM^{\dag}=A^{\dag}-U^{\dag}UM^{\dag}MA^{\dag}+U^{\dag}UM^{\dag}AA^{\dag}
+U^{\dag}AM^{\dag}MA^{\dag}-U^{\dag}AM^{\dag}AA^{\dag}=A^{\dag}-
(U^{\dag}UM^{\dag}-U^{\dag}AM^{\dag})(MA^{\dag}-AA^{\dag})=A^{\dag}-
(U^{\dag}(U-A)M^{\dag}(M-A)A^{\dag})=A^{\dag}-U^{\dag}VM^{\dag}NA^{\dag}=(I-H)A^{\dag}.$
Next to prove that $A=B-C$ is a proper splitting, i.e., to show that
$A=B-C$, $R(B)=R(A)$ and $N(B)=N(A)$. We have already shown the last
two conditions, so we have to prove only $A=B-C$. By Theorem
\ref{Main13}, we have $\rho(H)<1$ and so $I-H$ is invertible. Let
$X=A(I-H)^{-1}$. Then $XB^{\dag}=AA^{\dag}$ which results
$XB^{\dag}$ is symmetric and $XB^{\dag}X=X$. Again
$B^{\dag}X=(I-H)A^{\dag}A(I-H)^{-1}=(A^{\dag}A-HA^{\dag}A)(I-H)^{-1}=(A^{\dag}A-A^{\dag}AH)(I-H)^{-1}=A^{\dag}A$
which yields $B^{\dag}X$ is symmetric and
$B^{\dag}XB^{\dag}=A^{\dag}A(I-H)A^{\dag}=(I-H)A^{\dag}AA^{\dag}=B^{\dag}$.
Hence $X=(B^{\dag})^{\dag}=B=A(I-H)^{-1}$ and $C=B-A$. Now
$B^{\dag}C=B^{\dag}B-B^{\dag}A=B^{\dag}B -(I-H)A^{\dag}A=H$. Thus
$A=B-C$ is a proper splitting. Next, we have to prove that the
proper splitting $A=B-C$ is unique. Suppose that there exists
another induced splitting $A=\bar{B}-\bar{C}$ such that
$H=\bar{B}^{\dag}\bar{C}$. Then
$\bar{B}H=\bar{B}\bar{B}^{\dag}\bar{C}=\bar{C}=\bar{B}-A$.  So
$\bar{B}=A+\bar{B}H$, i.e., $\bar{B}(I-H)=A$. This reveals that
$\bar{B}=A(I-H)^{-1}=B$ and therefore, $H$ induces the unique proper
splitting $A=B-C$.

Finally, $B^{\dag}=U^{\dag}+U^{\dag}UM^{\dag}-U^{\dag}AM^{\dag}=
U^{\dag}+ (M^{\dag}-U^{\dag}AM^{\dag})=U^{\dag}+U^{\dag}VM^{\dag}\
\geq 0$ since $A=M-N=U-V$ are proper weak regular splittings and
$M^{\dag}-U^{\dag}AM^{\dag}= U^{\dag}VM^{\dag}$. Also
$B^{\dag}C=U^{\dag}V M^{\dag}N\geq 0.$ Hence $A=B-C$ with
$B=M(M+U-A)^{\dag}U$ is a proper weak regular splitting.
\end{proof}

Another question comes to picture now, i.e., among these splittings
which will converge faster. More specifically, we want to know the
rate of convergence of  the induced splitting for the iterative
scheme (\ref{eq3}). If the induced splitting $A=B-C$ will not
converge faster than the individual splittings $A=M-N$ and $A=U-V$,
then the proposed alternating iteration method will not be useful.
In this direction, we next present a result which compares the rate
of convergence  of the induced splitting with the individual
splitting.

\begin{thm}\label{Main33}
Let $A\in {\R}^{m\times n}$ and $A\geq 0$. Let $A=M-N=U-V$ be two
proper regular splittings of a semi-monotone matrix $A$ such that
$R(M+U-A)=R(A)$ and $N(M+U-A)=N(A)$. Then $\rho (H)\leq \mbox{min}
\{\rho(U^{\dag}V), \rho(M^{\dag}N)\} < 1$, where $H=U^{\dag}V
M^{\dag}N$.
\end{thm}

\begin{proof}
Let $H$ be the iteration matrix corresponding to the induced
splitting $A=B-C$. Then, by Theorem \ref{Main23}, $A=B-C$ is a
proper weak regular splitting. Using the conditions
$U^{\dag}VM^{\dag}\geq 0$ and $U^{\dag}NM^{\dag}\geq 0$, we have
$$B^{\dag}=U^{\dag}(M+U-A)M^{\dag}=U^{\dag}MM^{\dag}+U^{\dag}VM^{\dag}=
U^{\dag}+U^{\dag}VM^{\dag}\geq U^{\dag}$$ and
$$B^{\dag}=U^{\dag}(M+U-A)M^{\dag}=M^{\dag}+U^{\dag}NM^{\dag} \geq M^{\dag}.$$ Now, by Theorem
\ref{tcomp1} to the splittings $A=B-C$ and $A=U-V$, we have $$ \rho
(H)\leq \rho(U^{\dag}V)<1 .$$ Again, applying the same theorem to
the splittings $A=B-C$ and $A=M-N$, we obtain
$$ \rho (H)\leq \rho(M^{\dag}N) <1.$$
Hence $\rho (H)\leq \mbox{min} \{\rho(U^{\dag}V), \rho(M^{\dag}N)\}
< 1$.
\end{proof}

In words, Theorem  \ref{Main33} says that spectral radius of product
of iteration matrices $U^{\dag}V$ and $M^{\dag}N$ cannot exceed the
spectral radius of either factor under the assumption of some
conditions. The converse of Theorem \ref{Main33} does not hold. This
is illustrated by the following example.

\begin{ex}
Let $A= \begin{bmatrix}
                                                 \begin{array}{ccc}
                                                        1 & -2 & 3\\
                                                         2 & 3 & 4\\
                                                        \end{array}
                                                        \\
                                                        \end{bmatrix}$, $M=\begin{bmatrix}
                                                 \begin{array}{ccc}
                                                        1 & -2 & 3\\
                                                         -4 & -6  & -8\\
                                                        \end{array}
                                                        \\
                                                        \end{bmatrix}$
                                                        and $U=\begin{bmatrix}
                                                 \begin{array}{ccc}
                                                        3 & -6 & 9\\
                                                         5 & 15/2  & 10\\
                                                        \end{array}
                                                        \\
                                                        \end{bmatrix}$.
                                                        Then
                                                        $A=M-N=U-V$
                                                        are two
                                                        proper
                                                        splittings with $\rho (H)=\rho(U^{\dag}V
M^{\dag}N)=9/10<1$. But $A=M-N=U-V$ are not proper regular
splittings as $N=\begin{bmatrix}
                                                 \begin{array}{ccc}
                                                        0 & 0 & 0\\
                                                         -6 & -9  & -12\\
                                                        \end{array}
                                                        \\
                                                        \end{bmatrix}\ngeq 0$  and $V=\begin{bmatrix}
                                                 \begin{array}{ccc}
                                                        2 & -4 & 6\\
                                                         3 & 9/2  & 6\\
                                                        \end{array}
                                                        \\
                                                        \end{bmatrix}\ngeq 0$. Also $A\ngeq 0$.
\end{ex}

We next produce an example which states that the condition proper
regular splitting cannot be dropped.

\begin{ex}
Let $A= \begin{bmatrix}
                                                 \begin{array}{ccc}
                                                        1 & 0 & 0\\
                                                         0 & 0 & 0\\
                                                        \end{array}
                                                        \\
                                                        \end{bmatrix}$, $M=\begin{bmatrix}
                                                 \begin{array}{ccc}
                                                        2 & 0 & 0\\
                                                         0 & 0 & 0\\
                                                        \end{array}
                                                        \\
                                                        \end{bmatrix}$
                                                        and $U=\begin{bmatrix}
                                                 \begin{array}{ccc}
                                                        -1 & 0 & 0\\
                                                         0 & 0 & 0\\
                                                        \end{array}
                                                        \\
                                                        \end{bmatrix}$.
                                                        Then
                                                        $A=M-N=U-V$
                                                        are two
                                                           splittings with $\rho(H)=\rho(U^{\dag}V
M^{\dag}N)=1$. But $A=U-V$ is not a proper regular splitting as
$V=\begin{bmatrix}
                                                 \begin{array}{ccc}
                                                        -2 & 0 & 0\\
                                                         0 & 0  & 0\\
                                                        \end{array}
                                                        \\
                                                        \end{bmatrix}\ngeq
                                                        0$ . Then $\rho (H)=1\nleq \mbox{min}
                                                         \{\rho(U^{\dag}V)=2,
\rho(M^{\dag}N)=1/2\} \nless 1.$
\end{ex}

However, we have a few examples  which show that the Theorem
\ref{Main33} is also true even if  $A\ngeq 0$. One such example is
provided below.

\begin{ex}
Let $A= \begin{bmatrix}
                                                 \begin{array}{ccc}
                                                        2 & -1 & 0\\
                                                         -1 & 2 & 0\\
                                                        \end{array}
                                                        \\
                                                        \end{bmatrix}$, $M=\begin{bmatrix}
                                                 \begin{array}{ccc}
                                                        2 & -1 & 0\\
                                                         -1 & 3 & 0\\
                                                        \end{array}
                                                        \\
                                                        \end{bmatrix}$
                                                        and $U=\begin{bmatrix}
                                                 \begin{array}{ccc}
                                                        3 & -1 & 0\\
                                                         -1 & 3 & 0\\
                                                        \end{array}
                                                        \\
                                                        \end{bmatrix}$.
                                                        Then
                                                        $A=M-N=U-V$
                                                        are two
                                                        proper
                                                        regular
                                                           splittings with $\rho(H)=\rho(U^{\dag}V
M^{\dag}N)=7/40=0.175  \leq \mbox{min}
                                                         \{\rho(U^{\dag}V)=1/2=0.5,
\rho(M^{\dag}N)=2/5=0.4\}<1.$
\end{ex}

Note that Theorem \ref{Main33} also holds for $A=U-V$ is a proper
weak regular splitting. This suggests the following question.

\begin{center} {\it Can we drop the condition  $A\geq 0$ from Theorem
\ref{Main33} ?}
\end{center}

The answer is partially affirmative if we use of Theorem
\ref{tcomp2} in stead of Theorem \ref{tcomp1}. The same result is
stated below.

\begin{thm}\label{Main333}
Let  $A=M-N=U-V$ be two  proper regular splittings  of a
semi-monotone matrix $A$ such that $R(M+U-A)=R(A)$ and
$N(M+U-A)=N(A)$. Suppose that row sums of $U^{\dag}$ and $M^{\dag}$
are positive. Then $\rho (H)\leq \mbox{min} \{\rho(U^{\dag}V),
\rho(M^{\dag}N)\} < 1$, where $H=U^{\dag}V M^{\dag}N$.
\end{thm}

Finally we conclude this section with a problem which appears to be
open:

\begin{center}
 {\it Can we drop the conditions ``row sums of $U^{\dag}$
and $M^{\dag}$ are positive'' from Theorem \ref{Main333} ?}
\end{center}

\section{Conclusions}
The notion of the alternating iterative method for singular and
 rectangular linear systems is introduced. The present work extends the work
 of Benzi and Szyld \cite{benz} to rectangular(square singular) case.
The following three main results
 are obtained among others.

 \begin{itemize}

 \item Sufficient conditions for the convergence of alternating
 iteration scheme is provided (Theorem \ref{Main13}). This coincides
 with the first objective of Theorem 3.2, \cite{benz}  in case of nonsingular
 matrices.

 \item The induced splitting is shown to be a proper weak regular splitting under a few assumptions.
 This result not only partially fulfils the 2nd objective of Theorem 3.2, \cite{benz} in
 rectangular matrix setting but also extends Theorem 3.4, \cite{benz}.

\item Theorem
\ref{Main33} describes that the induced splitting is a better choice
among the individual splittings which  generalizes Theorem 4.1,
\cite{benz} for non-negative $A$.

\end{itemize}

The numerical benchmark of the alternating iterative method
indicates that the rate of convergence of the proposed  alternating
iterative method  is  not higher than the rate of convergence of the
usual iterative method. A problem for future study is also proposed
in the last part of Section 4. Not only that if we consider
$$X^{i+1}=U^{\dag}V M^{\dag}N X^{i}+U^{\dag}(VM^{\dag}+I),
~~~i=0,1,2,\cdots,$$ then this scheme will converge to the
Moore-Penrose inverse of $A$.

In case of a real square singular matrix,  let $m$ be the degree of
the minimal polynomial for $A$. If $b\in R(A^k)$, then the linear
system $Ax = b$ has a unique Krylov solution $x = A^Db \in
K_{m-k}(A, b)$, where $k$ is the index of $A$. Scope exists to
extend this work to compute $A^Db$, and the Drazin inverse of $A$
 as computing Drazin inverse of a matrix
is still a challenging problem.

\vspace{1cm}

\noindent {\small {\bf Acknowledgments.} The  author acknowledges
the support provided by Science and Engineering Research Board,
Department of Science and Technology, New Delhi, India, under the
grant number YSS/2015/000303.}

\bibliographystyle{amsplain}

\end{document}